\documentclass{amsart}
\usepackage{amssymb, amsmath, mathrsfs, verbatim, url}

\theoremstyle{plain}
\newtheorem{theorem}{Theorem}[section]
\newtheorem{lemma}[theorem]{Lemma}

\newtheorem{corollary}[theorem]{Corollary}

\theoremstyle{definition}

\newtheorem{question}[theorem]{Question}

\newcommand{\cl}{\mathsf{cl}}
\newcommand{\CH}{\mathsf{CH}}
\newcommand{\GCH}{\mathsf{GCH}}

\newcommand{\BB}{\mathcal{B}}

\newcommand{\NN}{\mathcal{N}}
\newcommand{\UU}{\mathcal{U}}
\newcommand{\VV}{\mathcal{V}}
\newcommand{\WW}{\mathcal{W}}
\newcommand{\PP}{\mathcal{P}}

\newcommand{\cccc}{\mathfrak{c}}

\begin{document}

\title{Productively Lindel\"of spaces of countable tightness}

\author{Andrea Medini}
\address{Kurt G\"odel Research Center for Mathematical Logic
\newline\indent University of Vienna
\newline\indent W\"ahringer Stra{\ss}e 25
\newline\indent A-1090 Wien, Austria}
\email{andrea.medini@univie.ac.at}

\author{Lyubomyr Zdomskyy}
\address{Kurt G\"odel Research Center for Mathematical Logic
\newline\indent University of Vienna
\newline\indent W\"ahringer Stra{\ss}e 25
\newline\indent A-1090 Wien, Austria}
\email{lzdomsky@gmail.com}


\thanks{The authors acknowledge the support of the FWF grant I 1209-N25. The second-listed author also thanks the Austrian Academy of Sciences for its generous support through the APART Program.}

\date{November 21, 2017}

\begin{abstract}
Michael asked whether every productively Lindel\"of space is powerfully
Lindel\"of. Building of work of Alster and De la Vega, assuming the Continuum Hypothesis, we show that every productively
Lindel\"of space of countable tightness is powerfully Lindel\"of. This
strengthens a result of Tall and Tsaban. The same methods also yield
new proofs of results of Arhangel'skii and Buzyakova. Furthermore, assuming the Continuum Hypothesis, we show that a productively
Lindel\"of space $X$ is powerfully Lindel\"of if every open cover of $X^\omega$ admits a point-continuum refinement consisting of basic open sets.
This strengthens a result of Burton and Tall.
Finally, we show that separation axioms are not relevant to Michael's question: if
there exists a counterexample (possibly not even $\mathsf{T}_0$), then there exists a regular
(actually, zero-dimensional) counterexample.
\end{abstract}

\maketitle

\section{Introduction}

The research in this article is
ultimately motivated by the following well-known question, which is credited to Michael by Alster (see \cite{alster}). Recall that a space $X$ is \emph{productively Lindel\"of} if $X\times Y$ is
Lindel\"of for every Lindel\"of space $Y$, and it is \emph{powerfully Lindel\"of} if $X^\omega$ is Lindel\"of. For all other notation and terminology, see Section 2.
\begin{question}[Michael]\label{qmichael}
Does productively Lindel\"of imply powerfully Lindel\"of?
\end{question}
Notice that if $X$ is productively Lindel\"of then $X^n$ is Lindel\"of for every $n\in\omega$. While, assuming $\CH$, there exists a non-powerfully Lindel\"of space $X$ such that
$X^n$ is Lindel\"of for every $n\in\omega$ (see \cite[Example 1.2]{michael}), Question \ref{qmichael} remains open under any set-theoretic assumption. The following seems to be the most substantial result on the subject (see \cite[Theorem 2]{alster}).
\begin{theorem}[Alster]\label{atheorem}
Assume $\CH$. If $X$ is a productively Lindel\"of space and
$w(X)\leq\cccc$ then $X$ is
powerfully Lindel\"of.
\end{theorem}
The technique of elementary submodels has been successfully employed by several
authors to establish further consequences of the above theorem. For example,
Burton and Tall obtained Theorem \ref{bttheorem} below, while Tall
and Tsaban obtained the following result (see \cite[Theorem 1.4]{talltsaban}).
\begin{theorem}[Tall, Tsaban]\label{tttheorem}
Assume $\CH$. If $X$ is a productively Lindel\"of sequential space then $X$ is
powerfully Lindel\"of.
\end{theorem}
Continuing in this tradition, we will show that Theorem \ref{tttheorem} can be improved by weakening ``sequential'' to
``of countable tightness'' (see Theorem \ref{mainct}). Furthermore, we will show that separation axioms are irrelevant to Question \ref{qmichael}, Theorem \ref{atheorem} and Theorem \ref{bttheorem} (see Corollary \ref{nosepmichael}, Theorem \ref{nosepalster} and Corollary \ref{nosepbt} respectively).
Finally, we will obtain a strengthening of Theorem \ref{bttheorem} (see Theorem \ref{mainpc}).

\section{Notation and terminology}

We will generally follow \cite{engelking}. In particular, every Lindel\"of space
is regular by definition. A non-empty space is \emph{zero-dimensional} if it is
$\mathsf{T}_1$ and it has a base consisting of clopen sets. It is easy to see that every
zero-dimensional space is regular (actually, Tychonoff). A space $X$ is \emph{quasi-Lindel\"of} if every open cover of $X$ has
a countable subcover. A space $X$ is \emph{productively quasi-Lindel\"of} if $X\times Y$ is quasi-Lindel\"of for every Lindel\"of space $Y$, and it is \emph{powerfully quasi-Lindel\"of} if $X^\omega$ is quasi-Lindel\"of. Given a space $X$ and $U\subseteq X^\omega$, we will say that $U$ is a \emph{basic open set}
if $U=\prod_{i\in\omega}V_i$, where each $V_i$ is an open subset of $X$
and $V_i=X$ for all but finitely many $i$.

The \emph{tightness} $t(X)$ of a space $X$ is the minimum cardinal $\kappa$ such that
whenever $x\in\cl(A)$ for some $A\subseteq X$ then there exists $B\in[A]^{\leq\kappa}$ such that $x\in\cl(B)$.
Given a subset $A$ of a space $X$, define $A_\alpha$ for $\alpha<\omega_1$ by recursion as follows.
\begin{itemize}
\item $A_0=A$.
\item $A_{\alpha+1}=\{x\in X: x\textrm{ is a limit of some sequence of elements
of }A_\alpha\}$.
\item $A_\gamma=\bigcup_{\alpha <\gamma}A_\alpha$, if $\gamma$ is a limit
ordinal.
\end{itemize}
A space $X$ is \emph{sequential} if $\cl(A)=\bigcup_{\alpha
<\omega_1}A_\alpha$ for every $A\subseteq X$. It is easy to see that every
sequential space has countable tightness.

The \emph{Lindel\"of number} $\ell(X)$ of a space $X$ is the least
cardinal $\kappa$ such that every open cover of $X$ has a subcover of size at
most $\kappa$. A family $\NN$ of subsets of a space $X$ is a
\emph{network} for $X$ if for every $x\in X$ and every neighborhood $U$ of $x$
there exists $N\in\NN$ such that $x\in N\subseteq U$. The \emph{network-weight}
$nw(X)$ of a space $X$ is the least cardinal $\kappa$ such that $X$ has a
network of size $\kappa$. The \emph{weight} $w(X)$ of a space $X$ is the least cardinal
$\kappa$ such that $X$ has a base of size $\kappa$. Given a cardinal $\kappa$ and a set $X$, a family $\WW$ of subsets of $X$
is \emph{point-$\kappa$} if $|\{W\in\WW:x\in W\}|\leq\kappa$ for every $x\in X$.

We will assume familiarity with the technique of elementary submodels (see for
example \cite{dow}). As usual, by ``elementary submodel'' we will really mean
``elementary submodel of $H(\theta)$ for a sufficiently large cardinal $\theta$''.
Given an infinite cardinal $\kappa$, an elementary submodel $M$ is \emph{$\kappa$-closed} if $[M]^{\leq\kappa}\subseteq M$. Given a set $S$ such that $|S|\leq 2^\kappa$, it is easy to
construct a $\kappa$-closed elementary submodel $M$ such that $S\subseteq M$
and $|M|=2^\kappa$.

\section{Adapting a method of De la Vega}

In this section, we adapt to our needs a method that De la Vega developed
in \cite{delavega} (see also \cite[Chapter 4]{delavegathesis}). In fact, Lemma
\ref{hausdorff} is \cite[Lemma 2.2]{delavega}, and the proof of Theorem \ref{lctimpliesnetwork} is inspired by the proof of \cite[Lemma 2.3]{delavega}.

\begin{lemma}[De la Vega]\label{hausdorff} Let $\kappa$ be an infinite cardinal. Assume that $(X,\tau)$ is a regular space such that $t(X)\leq\kappa$. Let $M$ be a $\kappa$-closed elementary submodel such that
$(X,\tau)\in M$, and let $Z=\cl(X\cap M)$. Then, whenever $z_0,z_1\in Z$ are distinct points, there exist
$U_0,U_1\in \tau\cap M$ such that $z_0\in U_0$, $z_1\in U_1$, and $U_0\cap
U_1=\varnothing$. 
\end{lemma}
\begin{proof}
Fix distinct $z_0,z_1\in Z$. Since $X$ is Hausdorff, we can fix $U_i\in \tau$ for $i\in 2$ such that $z_i\in U_i$ and $U_0\cap U_1=\varnothing$. Since $X$ is regular, we can fix $V_i\in \tau$ for $i\in 2$ such that $z_i\in V_i\subseteq \cl(V_i)\subseteq U_i$.
Since $t(X)\leq\kappa$, there exist $A_i\in [V_i\cap M]^{\leq\kappa}$ for $i\in 2$ such that $z_i\in\cl(A_i)$. Notice that each $A_i\in M$ because $M$ is $\kappa$-closed. Therefore
$$
M\vDash\textrm{There exist $U_0,U_1\in \tau$ such that $U_0\cap U_1=\varnothing$ and $\cl(A_i)\subseteq U_i$ for each $i$}
$$
by elementarity, which yields the desired $U_0,U_1$.
\end{proof}

\begin{theorem}\label{lctimpliesnetwork} Let $\kappa$ be an infinite cardinal. Assume that $(X,\tau)$ is a regular space with $\ell(X)\leq\kappa$ and $t(X)\leq\kappa$. Let $M$ be a $\kappa$-closed elementary submodel such that $(X,\tau)\in M$, and let $Z=\cl(X\cap M)$. Then $\{A\cap Z:A\in M\}$ is a network for $Z$.
\end{theorem}
\begin{proof}
Fix $z\in Z$ and assume that $z\in O\in\tau$. For each $x\in Z\setminus O$, use
Lemma \ref{hausdorff} to get $U_x,V_x\in\tau\cap M$ such that $x\in U_x$, $z\in
V_x$ and $U_x\cap V_x=\varnothing$.
Since $Z\setminus O$ is closed in $X$, there exists $C\in [Z\setminus O]^{\leq\kappa}$ such that $Z\setminus O\subseteq\bigcup\{U_x:x\in C\}$. Notice
that $\VV=\{V_x:x\in C\}\in M$ because $M$ is $\kappa$-closed. Hence
$A=\bigcap\VV\in M$ as well. The fact that $z\in A\cap Z\subseteq O$ concludes
the proof.
\end{proof}

\section{Countable tightness}

In this section, we give an affirmative answer to Question \ref{qmichael} for spaces of countable tightness (see Theorem \ref{mainct}). The main ingredients of the proof are Theorem \ref{lctimpliesnetwork} and Corollary \ref{network}. The following result first appeared as \cite[Lemma 3.3]{burtontall}. For a proof of a slightly more general result, see Corollary \ref{nosepbt}.

\begin{theorem}[Burton, Tall]\label{bttheorem}
Assume $\CH$. If $X$ is a productively Lindel\"of space such that
$\ell(X^\omega)\leq\cccc$ then $X$ is powerfully Lindel\"of.
\end{theorem}

\begin{corollary}\label{network}
Assume $\CH$. If $X$ is a productively Lindel\"of space with $nw(X)\leq\cccc$
then $X$ is powerfully Lindel\"of.
\end{corollary}

\begin{theorem}\label{mainct} Assume $\CH$. Let $(X,\tau)$ be a productively Lindel\"of space of countable tightness. Then $X$ is powerfully Lindel\"of.
\end{theorem}
\begin{proof}
Fix an open cover $\UU$ of $X^\omega$. Let $M$ be an $\omega$-closed elementary submodel such that $\{(X,\tau),\UU\}\subseteq M$ and $|M|=\cccc$. Let $Z=\cl(X\cap M)$.

First, we will show that $Z^\omega\subseteq\bigcup(\UU\cap M)$. So fix $z=(z_i:i\in\omega)\in Z^\omega$. Fix $U\in\UU$ such that $z\in U$. Let $V=\prod_{i\in\omega}V_i$ be such that $z\in V\subseteq\cl(V)\subseteq U$, where each $V_i\in\tau$ and $V_i=X$ for all but finitely many $i$. Given any $i\in\omega$, since $z_i\in\cl(V_i\cap M)$, we can fix $A_i\in[V_i\cap M]^{\leq\omega}$ such that $z_i\in \cl(A_i)$. Using the fact that $M$ is $\omega$-closed, it is easy to see that $A=\prod_{i\in\omega}A_i\in M$. Therefore
$$
M\vDash\textrm{There exists $U\in\UU$ such that $\cl(A)\subseteq U$}
$$
by elementarity, which yields $U\in\UU\cap M$ such that $z\in U$.

Observe that $Z$ is productively Lindel\"of because it is a closed subspace of $X$. Furthermore, it follows from Theorem \ref{lctimpliesnetwork} that $nw(Z)\leq |M|=\cccc$. Therefore $Z$ is powerfully Lindel\"of by Corollary \ref{network}, hence there exists $\VV\in[\UU\cap M]^{\leq\omega}$ such that $Z^\omega\subseteq\bigcup\VV$. Notice that $\VV\in M$ and $X^\omega\cap M=(X\cap M)^\omega\subseteq Z^\omega$ because $M$ is $\omega$-closed. It follows that
$$
M\vDash\textrm{$\VV$ is a cover of $X^\omega$}.
$$
Therefore, $\VV$ is a cover of $X^\omega$ by elementarity.
\end{proof}

\section{New proofs of results of Arhangel'skii and Buzyakova}

The following two results are \cite[Corollary 3.4]{arhangelskiibuzyakova} and \cite[Theorem 4.2]{arhangelskiibuzyakova}.
Recall that a space is \emph{linearly Lindel\"of} if it is regular and every open cover of $X$ that is linearly ordered by $\subseteq$ has a countable subcover.
\begin{theorem}[Arhangel'skii, Buzyakova]\label{ab1}
Assume $\CH$. Let $X$ be a Tychonoff space of countable tightness such that every open cover of $X$ of size at most $\omega_1$ has a countable subcover. Then $X$ is Lindel\"of.
\end{theorem}
\begin{theorem}[Arhangel'skii, Buzyakova]\label{ab2}
Assume $\GCH$. Let $X$ be a Tychonoff linearly Lindel\"of space such that $t(X)<\omega_\omega$. Then $X$ is Lindel\"of.
\end{theorem}

Using the same techniques as in the previous section, we will give new proofs of
the above results. In fact, it is clear that Theorem \ref{ab1} follows from Theorem \ref{ab1new} and that Theorem \ref{ab2} follows from Theorem \ref{ab2new}.
Notice that the assumption ``Tychonoff'' has been weakened to ``regular''.
\begin{theorem}\label{ab1new}
Let $\kappa$ be an infinite cardinal. Assume that $(X,\tau)$ is a regular space such that $t(X)\leq\kappa$ and every open cover of $X$ of size at most $2^\kappa$ admits a subcover of size at most $\kappa$. Then $\ell(X)\leq\kappa$.
\end{theorem}
\begin{proof}
Fix an open cover $\UU$ of $X$. Let $M$ be a $\kappa$-closed elementary submodel such that $\{(X,\tau),\UU\}\subseteq M$ and $|M|=2^\kappa$. Let $Z=\cl(X\cap M)$.
As in the proof of Theorem \ref{mainct}, one can show that $Z\subseteq\bigcup(\UU\cap M)$.
Since $|\UU\cap M|\leq |M|=2^\kappa$, there exists $\VV\in [\UU\cap M]^{\leq\kappa}$ such that $Z\subseteq\bigcup\VV$. As in the proof of Theorem \ref{mainct}, one sees that $\VV$ is a cover of $X$.
\end{proof}

\begin{theorem}\label{ab2new}
Assume that $2^\kappa<\omega_\omega$ for every $\kappa<\omega_\omega$. Let $(X,\tau)$ be a linearly Lindel\"of space such that $t(X)<\omega_\omega$. Then $X$ is Lindel\"of.
\end{theorem}
\begin{proof}
Fix an open cover $\UU$ of $X$. Let $\kappa<\omega_\omega$ be an infinite cardinal such that $t(X)\leq\kappa$. Let $M$ be a $\kappa$-closed elementary submodel such that $\{(X,\tau),\UU\}\subseteq M$ and $|M|=2^\kappa$. Let $Z=\cl(X\cap M)$.
As in the proof of Theorem \ref{mainct}, one can show that $Z\subseteq\bigcup(\UU\cap M)$. Using the fact that $X$ is linearly Lindel\"of, it is easy to see that every open cover of $X$ of size less than $\omega_\omega$ has a countable subcover.
Since $|\UU\cap M|\leq |M|=2^\kappa<\omega_\omega$, it follows that there exists $\VV\in [\UU\cap M]^{\leq\omega}$ such that $Z\subseteq\bigcup\VV$. As in the proof of Theorem \ref{mainct}, one sees that $\VV$ is a cover of $X$.
\end{proof}

\section{Dropping the separation axioms}

We will use the method of set-valued mappings introduced in \cite{zdomsky}.
Recall that a \emph{set-valued mapping} from a space $X$ to a space $Y$ is a
function $\Phi:X\longrightarrow\PP(Y)$, where $\PP(Y)$ denotes the power-set of
$Y$. A set-valued mapping from $X$ to $Y$ is
\emph{compact-valued} if $\Phi(x)$ is a compact subspace of $Y$ for every $x\in
X$. A set-valued mapping from $X$ to $Y$ is
\emph{upper semi-continuous} if $\{x\in X:\Phi(x)\subseteq V\}$ is open in $X$
for every open subset $V$ of $Y$. Given any set $S$, we will identify $2^S$ with $\PP(S)$ through characteristic functions. For $A\in 2^S$, let $A\!\uparrow\,\,=\{B\in 2^S:A\subseteq B\}$.

\begin{lemma}\label{nosepccover} Assume $\CH$. Let $X$ be a productively quasi-Lindel\"of space. Then every cover of $X^\omega$ of size $\cccc$ consisting of basic open sets has a countable subcover.
\end{lemma}
\begin{proof}
Let $\{U_\alpha:\alpha\in\kappa\}$ be a cover of $X^\omega$ consisting of basic open sets, where $\kappa=\cccc$. 
Write $U_\alpha=\prod_{i\in\omega}U_i^\alpha$ for each $\alpha$, where each $U_i^\alpha$ is an open subset of $X$ and $U_i^\alpha=X$ for all but finitely many $i$.

Consider the set-valued mapping from $X$ to $2^{\kappa\times\omega}$ obtained by
defining
$$
\Phi(x)=\{(\alpha,i)\in\kappa\times\omega:x\in U_i^\alpha\}\!\uparrow
$$
for every $x\in X$. Notice that $\Phi$ is compact-valued and
upper-semicontinuous.
It follows that $Y=\bigcup_{x\in X}\Phi(x)\subseteq 2^{\kappa\times\omega}$ is
productively Lindel\"of.
Since $\kappa=\cccc$, it is clear that $w(Y)\leq\cccc$. Therefore $Y$ is
powerfully Lindel\"of
by Theorem \ref{atheorem}.

For each $\alpha\in\kappa$, define
$$
V_\alpha=\{(y_i:i\in\omega)\in Y^\omega:(\alpha,i)\in y_i\textrm{ for every
}i\in\omega\}.
$$
We claim that $\{V_\alpha:\alpha\in\kappa\}$ is an open cover of $Y^\omega$.
First we will prove that $V_\alpha=\{(y_i:i\in\omega)\in Y^\omega:(\alpha,i)\in
y_i\textrm{ for every }i\in\omega\textrm{ such that }U^\alpha_i\neq X\}$. Notice that this implies
that each $V_\alpha$ is open. The inclusion $\subseteq$ is obvious. In order to prove the other inclusion, fix $y=(y_i:i\in\omega)\in Y^\omega$
such that $(\alpha,i)\in y_i$ for every $i\in\omega$ such that $U^\alpha_i\neq
X$. By the definition of $Y$, there exists $(x_i:i\in\omega)\in X^\omega$ such
that $y_i\supseteq\{(\beta,j)\in\kappa\times\omega:x_i\in U_j^\beta\}$ for each $i$.
We have to show that
$(\alpha,i)\in y_i$ for each $i$. So fix $i\in\omega$. If $U_i^\alpha\neq X$
then $(\alpha,i)\in y_i$ by assumption. On the other hand, if $U_i^\alpha= X$
then $(\alpha,i)\in\{(\beta,j)\in\kappa\times\omega:x_i\in U_j^\beta\}\subseteq
y_i$. Next, we will show that $\{V_\alpha:\alpha\in\kappa\}$ covers $Y^\omega$.
So fix $y=(y_i:i\in\omega)\in Y^\omega$.
By the definition of $Y$, there exists $x=(x_i:i\in\omega)\in X^\omega$ such
that $y_i\supseteq\{(\beta,j)\in\kappa\times\omega:x_i\in U_j^\beta\}$ for each
$i$.
Let $\alpha\in\kappa$ be such that $x\in U_\alpha$. It is clear that $y\in V_\alpha$.

To conclude the proof, assume that $S\subseteq\kappa$ is such that
$\{V_\alpha:\alpha\in S\}$ covers $Y^\omega$. It will be enough to show that
$\{U_\alpha:\alpha\in S\}$ covers $X^\omega$.
So fix $x=(x_i:i\in\omega)\in X^\omega$. Define
$y_i=\{(\beta,j)\in\kappa\times\omega:x_i\in U_j^\beta\}$ for each $i$, and
notice that each $y_i\in Y$.
Since $y=(y_i:i\in\omega)\in Y^\omega$, there exists $\alpha\in S$ such that
$y\in V_\alpha$. It follows from the definitions of $V_\alpha$ and $y_i$ that
$x\in U_\alpha$.
\end{proof}

Notice that the proof of Lemma \ref{nosepccover} also yields the following result. Corollary
\ref{nosepmichael} shows that separation axioms are irrelevant to Question \ref{qmichael}. The fact that separation axioms are irrelevant to the other, more
famous, question of Michael (whether $\omega^\omega$ is productively Lindel\"of)
was proved by Duanmu, Tall and Zdomskyy using the same methods (see \cite[Lemma 1]{duanmutallzdomskyy}).
\begin{theorem}\label{nosepalster}
Let $\kappa$ be an infinite cardinal. If there exists a productively quasi-Lindel\"of space $X$ with $w(X)\leq\kappa$ that is not powerfully quasi-Lindel\"of,
then there exists a zero-dimensional productively Lindel\"of space $Y$ with
$w(Y)\leq\kappa$ that is not powerfully Lindel\"of.
\end{theorem}

\begin{corollary}\label{nosepmichael}
The following are equivalent.
\begin{itemize}
\item Every productively quasi-Lindel\"of space is powerfully quasi-Lindel\"of.
\item Every productively Lindel\"of space is powerfully Lindel\"of.
\item Every zero-dimensional productively Lindel\"of space is powerfully
Lindel\"of.
\end{itemize}
\end{corollary}

\section{Point-$\cccc$ families}

In this section, we give an affirmative answer to Question \ref{qmichael} for one more class of spaces (see Theorem \ref{mainpc}). The main ingredients of the proof are Lemma \ref{nosepccover} and Lemma \ref{lemmapc}.

\begin{lemma}\label{lemmapc}
Let $X$ be a set, and let $\kappa$ be an infinite cardinal. Assume that $\WW$ is a point-$2^\kappa$ family of subsets of $X$. Let $M$ be a $\kappa$-closed elementary submodel such that $\{X,\WW\}\subseteq M$.
If $W\in\WW$ and $W\cap M\neq\varnothing$ then $W\in M$.
\end{lemma}
\begin{proof}
Define $\WW_x=\{W\in\WW:x\in W\}$ for $x\in X$, and notice that $|\WW_x|\leq 2^\kappa$ for every $x\in X$. Now fix $W\in\WW$ such that $W\cap M\neq\varnothing$. 
Let $z\in W\cap M$, and observe that $\WW_z\in M$. By elementarity,
$$
M\vDash\textrm{There exists a surjection $f:\PP(\kappa)\longrightarrow\WW_z$}.
$$
Furthermore, $\PP(\kappa)\subseteq M$ because $M$ is $\kappa$-closed. Therefore $\WW_z\subseteq M$, and in particular $W\in M$.
\end{proof}

\begin{theorem}\label{mainpc}
Assume $\CH$. Let $(X,\tau)$ be a productively quasi-Lindel\"of space such that every open cover of $X^\omega$ has a point-$\cccc$ refinement consisting of basic open sets.
Then $X$ is powerfully quasi-Lindel\"of.
\end{theorem}
\begin{proof}
It will be enough to show that every point-$\cccc$ cover of $X^\omega$ consisting of basic open sets has a countable subcover. So fix such a cover $\WW$. Let $M$ be an $\omega$-closed elementary submodel such that $\{(X,\tau),\WW\}\subseteq M$ and $|M|=\cccc$. Let $Z=\cl(X\cap M)$.

We claim that that $Z^\omega\subseteq\bigcup(\WW\cap M)$.
Fix $z\in Z^\omega$. Let $W\in\WW$ be such that $z\in W$. Using the fact that $M$ is $\omega$-closed, it is easy to check that $Z^\omega=\cl(X^\omega\cap M)$. Therefore $W\cap M\neq\varnothing$. Hence $W\in M$ by Lemma \ref{lemmapc}, which proves our claim.

Since $|\WW\cap M|\leq\cccc$, it follows from Lemma \ref{nosepccover} that there exists $\VV\in[\WW\cap M]^{\leq\omega}$ such that
$Z^\omega\subseteq\bigcup\VV$. Now proceed as in the proof of Theorem \ref{mainct}.
\end{proof}

The following corollary shows that Theorem \ref{mainpc} might be viewed as a strenghtening of Theorem \ref{bttheorem}.

\begin{corollary}\label{nosepbt}
Assume $\CH$. Let $X$ be a productively quasi-Lindel\"of space such that $\ell(X^\omega)\leq\cccc$.
Then $X$ is powerfully quasi-Lindel\"of.
\end{corollary}

As a further corollary of Theorem \ref{mainpc} one obtains that, under $\CH$, every productively Lindel\"of space with a point-$\cccc$ base is powerfully Lindel\"of,
which is a strengthening of Theorem \ref{atheorem}. However, as Corollary \ref{lindmiscenko} shows, the improvement is illusory. Although we could not find it in the literature, we feel that Theorem \ref{genmiscenko} might already be known. In fact, it is inspired by
the classical result of Mi\v{s}\v{c}enko stating that every compact space with a point-countable base has a countable base (see \cite{miscenko} or \cite[Exercise 3.12.23(f)]{engelking}), which can be proved using a similar argument (let $M$ be countable instead of $\kappa$-closed).

\begin{theorem}\label{genmiscenko}
Let $\kappa$ be an infinite cardinal. Assume that $(X,\tau)$ is a $\mathsf{T}_1$ space such that $\ell(X)\leq\kappa$ and $X$ has a point-$2^\kappa$ base. Then $w(X)\leq 2^\kappa$.
\end{theorem}
\begin{proof}
Fix a point-$2^\kappa$ base $\BB$ for $X$. Let $M$ be a $\kappa$-closed elementary submodel such that $\{(X,\tau),\BB\}\subseteq M$ and $|M|=2^\kappa$.
Define $\BB_x=\{B\in\BB:x\in B\}$ for $x\in X$, and notice that $|\BB_x|\leq 2^\kappa$ for every $x\in X$.
We claim that $X\cap M$ is dense in $X$. Since this implies $\BB=\bigcup_{x\in X\cap M}\BB_x$, hence $|\BB|\leq 2^\kappa$, this will conclude the proof.

Assume, in order to get a contradiction, that $z\in X\setminus\cl(X\cap M)$. Define
$$
\UU=\{B\in\BB:B\cap M\neq\varnothing\textrm{ and }z\notin B\},
$$
and notice that $\UU\subseteq M$ by Lemma \ref{lemmapc}. Using the fact that $\{z\}$ is closed, one sees that $\UU$ is a cover of $\cl(X\cap M)$.
Therefore, there exists $\VV\in[\UU]^{\leq\kappa}$ such that $\VV$ is a cover of $\cl(X\cap M)$. Observe that $\VV\in M$ because $\VV\subseteq\UU\subseteq M$ and $M$ is $\kappa$-closed, hence
$$
M\vDash\textrm{$\VV$ is a cover of $X$}.
$$
By elementarity, it follows that $\VV$ is a cover of $X$, contradicting our choice of $z$.
\end{proof}

\begin{corollary}\label{lindmiscenko}
Let $X$ be a Lindel\"of space with a point-$\cccc$ base. Then $w(X)\leq\cccc$.
\end{corollary}

\end{document}